\documentclass[10pt]{article}

\usepackage{amsmath, amsthm, amssymb, enumerate}
\usepackage{color}
\usepackage{datetime}
\usepackage{fancyhdr}

\chardef\coloryes=0 %%%out
\chardef\isitdraft=0 %%%out
\ifnum\isitdraft=1 %%%out
   \def\eqref#1{({\ref{#1}})}                %saves writing paranthesis%%%out

\usepackage[color,notcite]{showkeys}%%%out

\usepackage[notcite]{showkeys}%%%out

\definecolor{labelkey}{gray}{.3}%%%out

\definecolor{refkey}{rgb}{.3,0.3,0.3}%%%out

\else%%%out

  \def\startnewsection#1#2{\section{#1}\label{#2}\setcounter{equation}{0}}   %\starts a new section
  \def\nnewpage{} %\nnewpage does nothing
\fi%%%out

\begin{document}
%\newcommand{\llabel}{\label}             %\llabel is just a synonim for \label
%\newcommand{\rref}{\ref}                 %\rref is  just a synonim for \ref
%\newcommand{\ccite}{\cite}               %\ccite is just a synonim for
                                         %when ref. to equations
\def\ques{{\colr \underline{??????}\colb}}
\def\nto#1{{\colC \footnote{\em \colC #1}}}
\def\fractext#1#2{{#1}/{#2}}
\def\fracsm#1#2{{\textstyle{\frac{#1}{#2}}}}   %smaller version of frac
\def\nnonumber{}

%\newcommand{\bv}{{u}}

%if \isitdraft=0 or \coloryes=0%%%out
\def\colr{{}}%%%out
\def\colg{{}}%%%out
\def\colb{{}}%%%out
\def\cole{{}}%%%out
\def\colA{{}}%%%out
\def\colB{{}}%%%out
\def\colC{{}}%%%out
\def\colD{{}}%%%out
\def\colE{{}}%%%out
\def\colF{{}}%%%out
%\def\MR#1{}%%%out

%%3
%\ifnum\isitdraft=1%%%out
\ifnum\coloryes=1%%%out

  \definecolor{coloraaaa}{rgb}{0.1,0.2,0.8}%%%out
  \definecolor{colorbbbb}{rgb}{0.1,0.7,0.1}%%%out
  \definecolor{colorcccc}{rgb}{0.8,0.3,0.9}%%%out
  \definecolor{colordddd}{rgb}{0.0,.5,0.0}%%%out
  \definecolor{coloreeee}{rgb}{0.8,0.3,0.9}%%%out
  \definecolor{colorffff}{rgb}{0.8,0.3,0.9}%%%out
  \definecolor{colorgggg}{rgb}{0.5,0.0,0.4}%%%out

 \def\colg{\color{colordddd}}%%%out
 \def\colb{\color{black}}%%%out
 \def\colr{\color{red}}%%%out
 \def\cole{\color{colorgggg}}%%%out

 \def\colA{\color{coloraaaa}}%%%out
 \def\colB{\color{colorbbbb}}%%%out
 \def\colC{\color{colorcccc}}%%%out
 \def\colD{\color{colordddd}}%%%out
 \def\colE{\color{coloreeee}}%%%out
 \def\colF{\color{colorffff}}%%%out
 \def\colG{\color{colorgggg}}%%%out

%%4
\fi%%%out
%\fi%%%out
\ifnum\isitdraft=1%%%out
   \chardef\coloryes=1 %%%out
   \baselineskip=17pt%%%out
   \input macros.tex%%%out
   \def\blackdot{{\color{red}{\hskip-.0truecm\rule[-1mm]{4mm}{4mm}\hskip.2truecm}}\hskip-.3truecm}%%%out
   \def\bdot{{\colC {\hskip-.0truecm\rule[-1mm]{4mm}{4mm}\hskip.2truecm}}\hskip-.3truecm}%%%out
   \def\purpledot{{\colA{\rule[0mm]{4mm}{4mm}}\colb}}%%%out
   \def\pdot{\purpledot}%%%out
\else%%%out  
   \baselineskip=15pt
   \def\blackdot{{\rule[-3mm]{8mm}{8mm}}}%%%out
   \def\purpledot{{\rule[-3mm]{8mm}{8mm}}}%%%out
   \def\pdot{}
%%5
\fi%%%out

\def\tdot{\fbox{\fbox{\bf\tiny I'm here; \today \ \currenttime}}}
\def\nts#1{{\hbox{\bf ~#1~}}} %nts=note to self
\def\nts#1{{\colr\hbox{\bf ~#1~}}} %nts=note to self%%%out
\def\ntsf#1{\footnote{\hbox{\bf ~#1~}}} %nts=note to self
\def\ntsf#1{\footnote{\colr\hbox{\bf ~#1~}}} %nts=note to self%%%out
\def\bigline#1{~\\\hskip2truecm~~~~{#1}{#1}{#1}{#1}{#1}{#1}{#1}{#1}{#1}{#1}{#1}{#1}{#1}{#1}{#1}{#1}{#1}{#1}{#1}{#1}{#1}\\}%%%out
\def\biglineb{\bigline{$\downarrow\,$ $\downarrow\,$}}%%%out
\def\biglinem{\bigline{---}}%%%out
\def\biglinee{\bigline{$\uparrow\,$ $\uparrow\,$}}%%%out

\def\tilde{\widetilde}

\newtheorem{Theorem}{Theorem}[section]
\newtheorem{Corollary}[Theorem]{Corollary}
\newtheorem{Proposition}[Theorem]{Proposition}
\newtheorem{Lemma}[Theorem]{Lemma}
\newtheorem{Remark}[Theorem]{Remark}
\newtheorem{assumption}[Theorem]{Assumptions}

\newtheorem{definition}[Theorem]{Definition}
\def\theequation{\thesection.\arabic{equation}}
\def\endproof{\hfill$\Box$\\}
\def\square{\hfill$\Box$\\}
\def\comma{ {\rm ,\qquad{}} }            %comma in a formula
\def\commaone{ {\rm ,\qquad{}} }         %second comma in a formula
\def\dist{\mathop{\rm dist}\nolimits}    %distance
\def\sgn{\mathop{\rm sgn\,}\nolimits}    %sgn
\def\Tr{\mathop{\rm Tr}\nolimits}    %trace
\def\div{\mathop{\rm div}\nolimits}    %divergence
\def\supp{\mathop{\rm supp}\nolimits}    %divergence
\def\divtwo{\mathop{{\rm div}_2\,}\nolimits}    %two dimensional divergence
\def\re{\mathop{\rm {\mathbb R}e}\nolimits}    %distance

%added by ie
\def\dbC{\mathbb C}
\def\dbR{\mathbb R}
\def\dbL{\mathbb L}
\def\dbN{\mathbb N}
\def\dbP{\mathbb P}
\def\dbE{\mathbb E}
\def\dbF{\mathbb F}
\def\dbG{\mathbb G}
\def\dbQ{\mathbb Q}
\def\CC{{\tilde C}}
\def\cZ{{\cal Z}}
\def\cD{{\cal D}}
\def\e{\epsilon}
\def\d{\delta}
\def\1{{\mathbf{1}}}

\def\indeq{\qquad{}}                     %indentation in formulas
\def\period{.}                           %period in a formula
\def\semicolon{\,;}                      %semicolon in a formula
%**end of header

\title{Existence of invariant measures for the stochastic damped Schr\"odinger equation}
\author{Ibrahim Ekren, Igor Kukavica, and Mohammed Ziane}
\maketitle

\date{}

\begin{center}
\end{center}

\medskip

\indent Departement fur Mathematik, ETH Zurich, Ramistrasse 101, CH-8092, Zurich\\
\indent email: ibrahim.ekren@math.ethz.ch\\
\indent Department of Mathematics, University of Southern California, Los Angeles, CA 90089\\
\indent e-mails: kukavica\char'100usc.edu, ziane\char'100usc.edu

\begin{abstract}
In this paper, we address the long time behavior of solutions of the
stochastic Schr\"odinger equation in ${\mathbb R}^{d}$.
We prove the existence of an invariant measure and establish
asymptotic compactness of solutions, implying in particular
the existence of an ergodic measure.
\end{abstract}

\noindent\thanks{\em Mathematics Subject Classification\/}:
%35R35, %Free boundary problems
%35Q30, %Stokes and Navier-Stokes equations
%76D05  %Navier-Stokes equations 

\noindent\thanks{\em Keywords:\/}
Invariant measures, 
stochastic Schr\"odinger equation, 
white noise, 
long time behavior, 
asymptotic compactness, 
tightness, 
Feller property, 
Aldous criterion

\startnewsection{Introduction}{sec1}
The main purpose of the paper is to study the long time behavior of
the stochastic damped Schr\"odinger equation 
  \begin{equation}
    du + (\lambda u+i\Delta u -i|u|^{2\sigma} u )dt
    =  \Phi dW_t
   \label{EQ01}
  \end{equation}
in an unbounded domain.
Our main result 
provides the existence of an invariant measure
of the Markov semigroup for the equation \eqref{EQ01}
driven by an
additive noise.
In addition, using the asymptotic compactness, we prove that the set
of invariant measures is closed and convex leading to an existence of
an ergodic measure.

The problem of
existence of an invariant measure for stochastic 
partial differential equations 
with dissipation
and in a bounded domain
is now relatively well-understood with
the construction of the invariant measure following the classical
Krylov-Bogolyubov procedure. 
The smoothing
properties of the equation and the 
boundedness of the domain guarantee the necessary compactness.
For example, the existence of invariant measures for the reaction
diffusion equations, for the Navier-Stokes equations, 
complex Ginzburg-Landau, and
fractionally dissipated Euler equations was established in
\cite{CGV,F1,F2}.
Also, for the primitive equations, the invariant measure
was constructed in \cite{GKVZ}.

In the case of nondegenerate noise,
a coupling method can be used to establish
existence and uniqueness of the ergodic measure. For instance,
in the case of Schr\"odinger equation with
nondegenerate noise and when the domain is bounded, 
Debussche and Odasso
established in \cite{DO}
the existence of a unique ergodic measure
(cf.~also [DV,DZ,GMR,HM,KS,MR]).

The main goal of this paper is to address the existence of an
invariant measure  for the stochastic damped Schr\"odinger
equation in an unbounded domain.
The  main difficulties are the
the lack of smoothing and compactness properties of the solution operator in finite time.
For instance, the coupling method is not
expected to work in this situation
since Foias-Prodi type estimates, necessary for the
approach, are not available.  

In order to overcome these difficulties, we establish an asymptotic
compactness property of the solution operator
(cf.~Lemma~\ref{lemma-tightness}).  
Namely, we prove that for every sequence of
solutions resulting from $H^1$-bounded initial conditions and for
every sequence of times diverging to $\infty$, there exists a
subsequence of solutions and a sequence of times such that marginals of
these solutions at these times converge in distribution in $H^1$.  For
this purpose we employ the conserved quantities used classically for the
deterministic analog of the equations. We also use the energy
equation approach introduced in the deterministic setting case by
J.~Ball \cite{B}. His method was further developed
to more general deterministic situations, in particular to establish
the existence and regularity
of attractors for the damped KdV equation \cite{GR,R}
and for the damped
Schr\"odinger equation \cite{G1,G2,G3,GK,GL}.
Two byproducts of the asymptotic compactness property 
established in this paper
is the existence of an invariant measure for the stochastic Schr\"odinger equation and the compactness of the set of invariant measures.
We note that the existence and uniqueness of solutions was established
by de~Bouard and Debussche in \cite{DD}.

The paper is organized as follows.
In Section~\ref{sec4}, we prove an abstract tightness result that links the evolution of some scalar quantities to the asymptotic compactness stated above.
The main feature of the $k$-th order scalar quantity
is that it is equivalent to the $H^k$ norm, while the drift of 
square of its expectation is continuous in $H^{k-1}$ norm.
We also make an Aldous type continuity assumption (cf.~(iii) in Definition~\ref{energy-evolution}) which allows us
to use Aldous criterion \cite{Bi} for convergence of distributions in $L_{\rm loc}^{2}$
to pass to a limiting martingale solution \cite{D,MiR1,MiR2}.
We note that while the linear part is assumed to be a Schr\"odinger type
operator $i\Delta$, our criterion can be used for more general linear
operators as well
after suitable adjustments.
In Section~\ref{sec5}, we use this asymptotic compactness criterion
for the Schr\"odinger equation by considering the first two classical
Schr\"odinger invariants and 
prove the main  tightness lemma. 
The paper is concluded by showing that the set of invariant measures
is closed and convex, which implies the existence of an ergodic measure.

\startnewsection{Notations}{sec2}
For functions 
$u,v\in L^2(\dbR^{d})=L^2(\dbR^d;\mathbb{C})$, 
denote by 
$\Vert  u\Vert_{L^2}$ the $L^2(\dbR^d)$ 
norm of $u$ and by $(u,v)=\int_{\dbR^d} u(x)\overline v(x)dx$,
the
$L^2$-inner product of $u$ and $v$. We fix  a basis 
$\{e_i\}_{i\geq 0}$
of $L^2(\dbR^d)$ that consists of smooth and compactly supported functions.

For a Banach space $B$ and with $T>0$ and $p\ge1$, 
denote by $L^p([0,T];B)$ the space of functions from $[0,T]$ into $B$
with integrable $p$-th power over $[0,T]$ and by $C([0,T];B)$ the set
of continuous functions from $[0,T]$ into $B$. 
Similarly to functional spaces, for $p>0$, denote by $\dbL^p
(\Omega,B)$ the space of random variables with values in $B$ and a
finite $p$-th moment.

Denote by $\Delta =\sum_i\partial^2_{i}$ the Laplace operator and by $H^r (\dbR^d)$ the Sobolev space of  functions $u$ satisfying
  \begin{equation}\Vert u\Vert^2_{H^r}
        =\int_\dbR (1-\Delta)^{\fractext{r}{2}} 
              (u(x)\overline u(x)) dx<\infty,
   \end{equation}
with the inner product denoted by $(u,v)_{H^r}$. 
Write ${\cal B} 
  (H^1(\dbR^d) )$ 
for the set of Borel measurable subsets of $H^1(\dbR^d)$.
Also, denote by $L^2_{\rm loc}(\dbR^d)$ the space of locally square
integrable functions which with the usual metric is a
complete
metric space.

For a Hilbert space $H$, we write ${{\rm HS}(L^2,H)}$ for the space 
of linear operators $\Phi\colon L^2(\dbR^d) \to H$ 
with finite Hilbert-Schmidt norm 
  \begin{equation}
     \Vert\Phi\Vert_{{\rm HS}(L^2,H)}
        =
         \left(
          \sum_{i=1}^\infty \Vert\Phi e_i\Vert_H^2
         \right)^{1/2}
   \period
  \end{equation}
%and $W$ a cylindrical Wiener process on $L^2(\dbR)$ under the stochastic basis.  

\startnewsection{The Schr\"odinger equation}{sec3} 
We fix a probability space $(\Omega, \dbF,\dbP)$ carrying a countable
family of independent Brownian motions 
$\{B^i_t\}_{i\in\dbN,t\geq 0}$ and define the Wiener process 
  \begin{equation}
    W_t=\sum_{i\in\dbN}e_i B_t^i\period
  \end{equation} 
Fix $\lambda>0$. In this paper,
we investigate the long time behavior
of solutions of the stochastic damped nonlinear Schr\"odinger equation 
  \begin{equation}\label{equation-schro}
    du + (\lambda u+i\Delta u -i|u|^{2\sigma} u )dt
    =  \Phi dW_t,
  \end{equation}
on the space-time 
domain $[0,\infty)\times \dbR^d$ 
with an additive noise, by establishing
the existence of an invariant measure and the asymptotic 
tightness of 
solutions of 
the equation.
%A well-known procedure is the
%use of the Krylov-Bogoliubov theorem which relies on two important
%features of an SPDE; namely the Feller property of the semigroup and
%the tightness for integral time averages.
We emphasize that unlike in \cite[Assumption~H1]{MR}, our problem
in the whole space $\dbR^d$ does not allow any compact embeddings. 

Recall the functionals 
  \begin{align}
   M(v)&=|v|^2_{L^2}\\
   H(v)&=\frac{1}{2} \int_{\dbR^d} |\nabla v(x)|^2 dx -\frac{1}{2\sigma +2}\int_{\dbR^d}|v(x)|^{2\sigma +2}dx
  \end{align}
which are 
classical invariant quantities for the Schr\"odinger equation.
The existence of solutions for the equation \eqref{equation-schro} was
proven in \cite{DD}.
% We make the following assumptions to be able to
%apply their existence result. 
In order to be able to apply the existence results in
\cite{DD}, we make the following assumptions.

\begin{assumption}\label{assumption-existence-schro}
i) $0 \leq \sigma <\fractext{2}{(d-2)}$ if $d\geq 3$ or $\sigma\ge0$ 
if $d=1,2$.\\
ii) $\Phi\in HS(L^2(\dbR^d); H^1(\dbR^d))$.
\end{assumption}

We now 
recall the existence result from \cite[Theorem~3.4, Propositions~3.2 and~3.4]{DD}. 

\cole
\begin{Theorem}
Under Assumptions~\ref{assumption-existence-schro}, for every
$\dbF_0$ measurable, $H^1(\dbR^n)$ valued random variable $u_0$, 
there exists an $H^1(\dbR^d)$-valued 
and continuous solution $\{u_t\}_{t\geq 0}$ 
of \eqref{equation-schro} with the initial condition $u_0$. 
Additionally, the  quantities $M$ and $H$ evolve as
  \begin{align}\label{evolution-M}
   dM(u_s)
    +2\lambda  M(u_s)ds =2 \sum_{i} Re \left( u_s,{\Phi e_i}\right)dB^i(s)+ \Vert\Phi \Vert_{HS(L^2; L^2)}^2ds 
\end{align}
and
  \begin{align}\label{evolution-H}
  &dH(u_s)+2\lambda H(u_s)ds \notag\\
   \indeq  &=
  	\frac{\lambda\sigma}{\sigma+1}  \int |u(\tilde s,x)|^{2\sigma+2}dx d\tilde s
         - \sum_i \re \left( \Delta u(s) +|u(s)|^{2\sigma} u(s),\Phi e_i
          \right)dB^i_s\notag
     \nonumber\\&\indeq
       +\left(\frac{\Vert\nabla \Phi \Vert^2_{{HS(L^2; L^2)}}}{2} 
       -\frac{\Vert|u(s)|^\sigma\Phi\Vert^2_{{HS(L^2;
       L^2)}}}{2}\right)ds
     \nonumber\\&\indeq
      -{\sigma}\sum_i  \left( |u(s)|^{2\sigma-2},(\re(\overline u(s) \Phi e_i))^2\right)ds,
  \end{align}
where $|u(s)|^\sigma\Phi$ is the operator that to a function $v$ associates the function $|u(s)|^\sigma \Phi v$.
\end{Theorem}
\colb

Note that the results in \cite{DD} are given for $\lambda =0$ but one
can easily pass from $\lambda=0$ to any $\lambda> 0$.

\subsection{The Semigroup}
Let $u_0\in H^1(\dbR^d)$ be a deterministic initial condition, and let
$u$ be the corresponding solution of \eqref{equation-schro}. For all $B\in {\cal B} (H^1(\dbR^d) )$ we define the transition probabilities of the equation by 
  \begin{equation}\label{transition-prob}
    P_t (u_0,B)=\dbP(u_t\in B)
   \period
  \end{equation}
For any $H^1(\dbR^d)$-valued measure $\nu$, we denote by 
$(\nu P_t)(\cdot)=\int_{H^1(\dbR^d)} P_t(v,\cdot)\nu (dv)$ the distribution
at time~$t$ of the solution of \eqref{equation-schro} with the initial condition having the distribution $\nu$.
 
For any function $\xi\in C_b(H^1(\dbR^d);\dbR)$ and $t\geq 0$, denote 
  \begin{equation}\label{definition-psi}
    P_t \xi(u_0) =\dbE\left[\xi(u_t)\right]=\int_{H^1(\dbR^d)} \xi(v)P_t(u_0, dv)
   \period
  \end{equation}

\begin{definition}
{\rm Let $\mu$ be a probability measure on $H^1(\dbR^d)$. We say that $\mu$
is an invariant measure for $P_t$ if we have
  \begin{align}
  \int_{H^1(\dbR^d)} \xi(v)\mu( dv)=\int_{H^1(\dbR^d)} P_t\xi(v') \mu( dv')
  \end{align}
for all $\xi\in C_b(H^1(\dbR^d);\dbR)$ and $t\geq 0$.
}
\end{definition}

\subsection{Main results concerning the Schr\"odinger equation}
The following statement is the main result of this paper.

\cole
\begin{Theorem}
\label{T01}
Under Assumptions~\ref{assumption-existence-schro}, 
there exists an invariant measure for $P_t$. 
\end{Theorem}
\colb

The main ingredient in the proof is the following lemma.

\cole
\begin{Lemma}
\label{lemma-tightness}
Under Assumptions~\ref{assumption-existence-schro} the
following two tightness assertions hold.\\
i) For all sequences
of times $t_n\to\infty$ and 
$\dbF_0$-measurable initial conditions
$u^n_0\in H^1(\dbR^d)$ with distributions 
$\nu^n$ satisfying 
  \begin{equation*}
     \dbE\left[\Vert u^n_0\Vert _{H^1}^{4\vee\lceil{4d\sigma}\rceil}+\Vert u^n_0\Vert _{L^2}^{\lceil{4\sigma(2-d)+8}\rceil}\right]\leq R
  \end{equation*}
for some $R>0$,  
the family 
of measures
  \begin{equation}
    \bigl\{
       (\nu^n P_{t_n})(\cdot)
       :
       n\in\dbN
    \bigr\}
   \label{EQ02}
  \end{equation}
on $H^{1}({\mathbb R^d})$ 
is tight.\\
ii) For all compact sets $K\subseteq H^1(\dbR^d)$ the family of probabilities 
  \begin{equation}
    \left\{P_{s}(v,\cdot):s\in[0,1],\,v\in K\right\}
   \label{EQ03}
  \end{equation}
on $H^1(\dbR^d)$    
is tight.
\end{Lemma}
\colb

Assuming the lemma, 
we now prove the main theorem. The lemma
is then proven in Section~\ref{sec5} below.

\begin{proof}[Proof of Theorem~\ref{T01}]
An invariant measure is constructed using the classical
Krylov-Bogolyubov theorem, which requires the Feller property of the semigroup and the tightness of averaged measures 
  \begin{equation}
    \mu_n(\cdot):=\frac{1}{n}\int_0^n P_t(0,\cdot) dt\period
  \end{equation} 
The Feller property is a consequence of \cite[Proposition~3.5]{DD}. 
%Assuming Lemma~\ref{lemma-tightness}, we now prove the tightness of
% the family $\mu_n$.
Thus in order to conclude the proof, we only need to 
show tightness of the family of measures $\mu_{n}$.

Let $\epsilon>0$. Lemma~\ref{lemma-tightness} applied to the family
$\{P_{k}(0,\cdot);k\in \dbN\}$ gives the existence of a compact set
$K_\epsilon\subseteq H^1(\dbR^d)$ such that   
  \begin{equation}
     \sup_{k} P_k(0,K_\epsilon^c)
     \leq 
     \frac{\epsilon}{2}
   \period
  \end{equation}
We then consider the family of probabilities $\{P_s(v,\cdot): s\in[0,1],\, v\in K_\epsilon\}$. By the second part of Lemma~\ref{lemma-tightness}, this family is tight. 
Therefore, there exists another compact set $A_\epsilon\subseteq H^1(\dbR^d)$ such that 
  \begin{equation}
    \sup_{ s\in[0,1], v\in K_\epsilon} P_s(v,A_\epsilon^c)
    \leq 
    \frac{\epsilon}{2}
   \period
  \end{equation}
By a direct computation 
\begin{align}
\mu_n(A^c_\epsilon)&=\frac{1}{n}\int_0^n P_t(0,A_\epsilon^c) dt   = \frac{1}{n}\sum_{k=0}^{n-1} \int_{k}^{k+1} P_t(0,A_\epsilon^c) dt\\
&=\frac{1}{n}\sum_{k=0}^{n-1} \int_{k}^{k+1} \int_{H^1(\dbR^d)}P_{k} (0,dv) P_{t-k}(v,A_\epsilon^c) dt\\
&=\frac{1}{n}\sum_{k=0}^{n-1} \int_{k}^{k+1} \left(\int_{H^1(\dbR^d)\cap K_\epsilon^c}P_{k} (0,dv) P_{t-k}(v,A_\epsilon^c) +\int_{H^1(\dbR^d)\cap K_\epsilon}P_{k} (0,dv) P_{t-k}(v,A_\epsilon^c) \right)dt
\end{align}
whence
\begin{align}
\mu_n(A^c_\epsilon)
&\leq \frac{1}{n}\sum_{k=0}^{n-1} \left(P_{k} (0,K_\epsilon^c)  +P_{k} (0,K_\epsilon) \sup_{ s\in[0,1],\, v\in K_\epsilon} P_s(v,A_\epsilon^c)\right) \\
&\leq \frac{1}{n}\sum_{k=0}^{n-1} \left(P_{k} (0,K_\epsilon^c)  +\sup_{ s\in[0,1],\, v\in K_\epsilon} P_s(v,A_\epsilon^c)\right) \leq \epsilon
   \period
\end{align}
We have thus shown that the set of measures $\{\mu_n\}$ is
tight, concluding the proof of the theorem.
%relatively weakly
% compact which is the required tightness result for the
% Krylov-Bogolyubov procedure.
\end{proof}

The rest of the paper is devoted to the proof of 
Lemma~\ref{lemma-tightness} and to establishing
the compactness of the set of invariant measures.

\startnewsection{An abstract tightness result}{sec4}
In this section, we give certain distributional convergence results that
we use below to prove Lemma~\ref{lemma-tightness}. 

\cole
\begin{Lemma}\label{cv-loc-hi} 
Let $k\in{\mathbb N}_0$, and
let $\xi_n$ and $\xi$ be an $H^k(\dbR^d)$-valued square integrable random variables such that $\xi_n \to \xi$ in distribution in $L^2_{\rm loc}(\dbR^d)$. 
Assume that $\dbE[\Vert \xi_n\Vert^2_{H^k}]\to \dbE[\Vert \xi\Vert^2_{H^k}]$ as $n\to \infty$ and suppose
that the family 
$\{\Vert \xi_n\Vert^2_{H^k}:n\in{\mathbb N}\}$
is uniformly integrable.
Then $\xi_n$ converges to $\xi$ in distribution in $H^k(\dbR^d)$. 
\end{Lemma}
\colb

Note that when $k=0$, we have
$H^0(\dbR^d)=L^2(\dbR^d)$.

\begin{proof}[Proof Lemma~\ref{cv-loc-hi}]
 Let $\{f_i\}$ be a complete orthonormal system for $H^k(\dbR^d)$
 consisting of smooth compactly supported functions.
We first claim that
  \begin{equation}\label{limit-Pro}
   \lim_{N\to \infty }\sup_n \dbE\left[\sum_{i=N}^\infty |(\xi_n,f_i)_{H^k}|^2\right]=0
  \end{equation}
which then quickly implies asserted convergence.
Let $\epsilon>0$.
By the uniform integrability assumption, there exists $R>0$ such that 
  \begin{equation}\sup_n \dbE\left[ \Vert \xi_n\Vert^2_{H^k}\1_{ \{\Vert \xi_n\Vert^2_{H^k}\geq R\}}\right] \leq \epsilon
  \end{equation}
and, by possibly enlarging $R$, we may also assume that
  \begin{equation}\dbE\left[ \Vert \xi\Vert^2_{H^k}\1_{ \{\Vert \xi\Vert^2_{H^k}\geq R\}}\right] \leq \epsilon\period
  \end{equation}
%For all $N$ the convergence in distribution in $L^{2}_{\rm loc}(\dbR)$ and the fact that $e_i$ is smooth and compactly supported allow us to conclude that the 
For all $N\in{\mathbb N}$, the convergence in distribution in $L^2_{\rm loc}(\dbR^d)$ and the fact that $\{f_i\}$ have compact support imply
   \begin{equation}
   \dbE\left[\left(\sum_{i=1}^N |(\xi_n,f_i)_{H^k}|^2\right)\wedge R\right]\to\dbE\left[\left(\sum_{i=1}^N |(\xi,f_i)_{H^k}|^2\right)\wedge R\right]
  \end{equation}
  as $n\to\infty$.
Since
  \begin{align}
     &\left|\dbE\left[\sum_{i=1}^N |(\xi_n,f_i)_{H^k}|^2\right] -\dbE\left[\sum_{i=1}^N |(\xi,f_i)_{H^k}|^2\right]\right|
   \nonumber\\&\indeq
    \leq \left|\dbE\left[\left(\sum_{i=1}^N |(\xi_n,f_i)_{H^k}|^2\right)\wedge R\right]-\dbE\left[\left(\sum_{i=1}^N |(\xi,f_i)_{H^k}|^2\right)\wedge R\right]\right|
   \nonumber\\&\indeq\indeq
+ \dbE\left[ \Vert \xi_n\Vert^2_{H^k}\1_{ \{\Vert \xi_n\Vert^2_{H^k}\geq R\}}\right] + \dbE\left[ \Vert \xi\Vert^2_{H^k}\1_{ \{\Vert \xi\Vert^2_{H^k}\geq R\}}\right]
  \end{align}
we have that
  \begin{equation}\lim_n \dbE\left[\sum_{i=1}^N |(\xi_n,f_i)_{H^k}|^2\right] =\dbE\left[\sum_{i=1}^N |(\xi,f_i)_{H^k}|^2\right]\period
  \end{equation}
This convergence combined with the assumption $\dbE[\Vert \xi_n\Vert^2_{H^k}]\to \dbE[\Vert \xi\Vert^2_{H^k}]$ implies 
  \begin{equation}\label{convergence-remainder}
   \dbE\left[\sum_{i=N+1}^\infty |(\xi_n,f_i)_{H^k}|^2\right]\to\dbE\left[\sum_{i=N+1}^\infty |(\xi,f_i)_{H^k}|^2\right]
   \period
  \end{equation}
Since $\xi$ is $H^k(\dbR^d)$-square integrable,
there is $N_0\in{\mathbb N}_0$ such that 
  \begin{equation}
   \dbE\left[\sum_{i=N_0+1}^\infty |(\xi,f_i)_{H^k}|^2\right]
   \leq 
    \frac{\epsilon}{2}
   \period
   \label{EQ26}
  \end{equation}
Then, using \eqref{convergence-remainder}, there exists
$n_\epsilon\in{\mathbb N}$ 
for which
  \begin{equation}
   \sup_{n\geq n_\epsilon}
       \dbE\left[\sum_{i=N_0+1}^\infty |(\xi_n,f_i)_{H^k}|^2\right]\leq \epsilon   
   \period
   \label{EQ27}
  \end{equation}
The family $\{\Vert\xi_n\Vert_{H^k}: n=1,\ldots, n_\epsilon-1\}$ is square integrable. Therefore,
%For all $n\leq n_\epsilon-1$, 
  \begin{equation}
   \lim_{N\to \infty}
      \dbE\left[\sum_{i=N}^\infty |(\xi_n,f_i)_{H^k}|^2\right]=0   
   \comma n\leq n_{\epsilon}-1
   \period
   \label{EQ29}
  \end{equation}
By \eqref{EQ27} and \eqref{EQ29}, there exists
$N_1\geq N_0$ such that
  \begin{equation}
    \sup_{n\in{\mathbb N}}
    \dbE\left[\sum_{i=N_1+1}^\infty |(\xi_n,f_i)_{H^k}|^2\right]\leq \epsilon
   \period
   \label{EQ33}
  \end{equation}
Therefore, \eqref{limit-Pro} is established.

By  \cite[Theorem~1.13]{Pr}, 
the convergence \eqref{limit-Pro} then
implies the tightness in distribution in $H^k(\dbR^d)$ of the laws of $\{\xi_n\}$. Note that any limiting measure can only be the distribution of $\xi$. Thus 
  \begin{equation}\label{limit-l2}
   \xi_n\to \xi
  \end{equation}
 in distribution in $H^k(\dbR^d)$. 
\end{proof}

We shall work on the space 
$\cZ=C([0,T]; L^{2}_{\rm loc}(\dbR))$.
Denote by $z$ the canonical process on this space 
and $\cD$ its right continuous filtration. 
We state our main theorem for an SPDE of the form 
  \begin{align}\label{EQ00}
   &du(t) = \bigl(
              -i \Delta u(t)+b(u(t))
            \bigr)dt +\Phi dW_t
  \end{align}
with
  \begin{align}
   &u(t)\in L^2_{\rm loc}(\dbR)\notag
  \end{align}
where $b\colon \dbC \to \dbC$
is, for simplicity,
a sum of terms of the form $u^{m} |u|^{a}$ where $m\in{\mathbb N_0}$ and 
$a\ge0$.
Its maximal degree ($\max\{m+a\}$) is assumed to be less than 
  \begin{equation*}
   \frac{2 d}{d-2k}
  \end{equation*}
if $d>2k$.
%$n/(n-2k)$ if $n\ge 2k+1$
%(there is no restriction on the degree if $n\ge 2k$).

\begin{definition}\label{defn-weak-sol}
{\rm A measure $\nu$ on $\cZ$  is a martingale solution of the
equation \eqref{EQ00} if for all $\phi$ smooth and compactly
supported functions
  \begin{equation}
    \int_0^T \big(|b(z_s)|,|\phi|\bigr) ds <\infty, \, \nu \mbox{-a.s.}
  \end{equation}
and  if
  \begin{equation}
      M^\phi_t
         =(z_t-z_0,\phi) -\int_0^t (-i\Delta z_s +b(z_s),\phi)ds
                                              \label{EQ38}
  \end{equation}
and 
  \begin{equation}
    (M^\phi_t)^2-\int_0^t \sum_i \left(\Phi e_i,\phi\right)^2ds
  \label{EQ48}
  \end{equation}
are $\nu$-local martingales.
We say that $\nu$ is a $H^k$ square integrable martingale solution if 
  \begin{equation}\sup_{t\in [0,T]} \dbE^\nu \left[\Vert z_t\Vert_{H^k}^2\right]<\infty\period
  \end{equation}
}
\end{definition}

\begin{Remark}{\rm
i) Note that a martingale solution can be obtained from any strong solution of \eqref{EQ00}. Indeed, let $u$ be a solution of \eqref{EQ00} on the interval $[0,T]$. 
Define the measure 
  \begin{equation}
      \nu(dz) 
        =\int_\Omega \delta_{\{\{u_{s}(\omega)\}_{s\in [0,T]}\}}(dz)\dbP(d\omega)
   \label{EQ39}
  \end{equation}
meaning the measure on $\cZ$ such that for all continuous bounded $F\colon\cZ\to\dbR$
we have
    \begin{equation}
      \int_\cZ F(z)\nu (dz)= \dbE[F(\{u(s)\}_{s\in[0,T]})]
   \period
    \end{equation}}
\end{Remark}

In order to facilitate the statement of the main result of this section, 
we introduce the concept of $H^{k}$-evolution property.

\begin{definition}\label{energy-evolution}
{\rm Let $k\in{\mathbb N}$.
The equation \eqref{EQ00} has the $H^k$-norm evolution property if  
for $i=0,\ldots, k$ 
there exist continuous functions
$F_i\colon L^2_{\rm loc}\to \dbR$, 
$\tilde F_i \colon\dbR \times L^2_{\rm loc}\to \dbR$, 
and $  G_i \colon\dbR \times \dbR \times L^2_{\rm loc}\to \dbR$
satisfying the following conditions.\\
i) For all $t,r$ the functions $F_i(\cdot),\tilde F_i(t,\cdot)$ and $G_i(t,r,\cdot)$ are continuous in $H^{i-1}$-topology on bounded sets of $H^{i}$ (for $i=0$ we require the continuity in $L^2_{\rm loc}$ on bounded sets of $L^2$).\\
ii) For all $t,r$ the functions $F_i(\cdot),\tilde F_i(r,\cdot)$ and $G(r,\cdot)$  have at most polynomial growth in the $H^i$-norm. \\
iii) For all $H^k$ square integrable martingale solutions $\nu$ of \eqref{EQ00}, the conservation equality 
  \begin{equation}\label{conserved-quantities}
   \dbE^\nu\left[\Vert z_t\Vert^2_{H^i}\right]-e^{-2\lambda (t-s)}\dbE^\nu\left[\Vert z_s\Vert^2_{H^i}\right]=\dbE^\nu\left[F_i(z_t)-\tilde F_i(t-s,z_s)\right]
   +\int_s^t \dbE^\nu\left[G_i(t-r,z_r)\right]dr
  \end{equation}
holds for all $t\ge s$.
}
\end{definition}

We now state a theorem, 
which, combined with Lemma~\ref{cv-loc-hi}, 
gives us a tightness result needed for 
the Krylov-Bogolyubov procedure. 

\cole
\begin{Theorem}\label{thm-limit-conservation}
Assume that the equation \eqref{EQ00} 
has the $H^k$-norm evolution property,
and let $u^n$ be a sequence of 
strong solutions of \eqref{EQ00} satisfying
the following conditions:\\
a) We have a uniform bound 
  \begin{equation}\label{unif-bound-theo}
   \gamma=
  \sup_{r\geq 0}\sup_{k\geq i\geq 0} \sup_{n,t} \dbE\left[\Vert b(u^n_t)\Vert_{L^1}^2+| F_i(u^n_t)|^2+|\tilde F_i(r,u^n_t)|^2+|G_i(r,u^n_t)|^2+ \Vert u^n_t\Vert^4_{H^k}\right]<\infty
   \period
\end{equation}
b) For every sequence of stopping times $T_n$ and positive numbers $\d_n$ such that $\d_n\to 0$ as $n\to \infty$, we have
  \begin{equation}\label{aldous-criterion}
   \dbE\left[\Vert u^n_{T_n+\d_n}-u^n_{T_n}\Vert^2_{L^2}\right]\to 0\mbox{~~as~} n\to \infty
   \period
  \end{equation} 
c) There exists a sequence 
$t_n\to \infty $ and an $H^k$-valued random variable 
$\xi$ such that 
$u^n_{t_n}\to \xi$ in distribution in $L^2_{\rm loc}$. \\
Then $\dbE[\Vert u^n_{t_n}\Vert^2_{H^k}]\to \dbE[\Vert \xi\Vert^2_{H^k}]$ as $n\to \infty$.
\end{Theorem}
\colb

\begin{Remark}{\rm
The powers in \eqref{unif-bound-theo} 
have been chosen so we can obtain the uniform
integrability of the family and then the De la Vallee Poussin's 
theorem can be applied.}
\end{Remark}

\begin{proof}[Proof of Theorem~\ref{thm-limit-conservation}] Since
    
  \begin{equation}\liminf_n \dbE[\Vert u^n_{t_n}\Vert^2_{H^k}]\geq\dbE[\Vert\xi\Vert^2_{H^k}]
  \end{equation}
we only need to 
prove 
  \begin{equation}
   \limsup_n \dbE[\Vert u^n_{t_n}\Vert^2_{H^k}]\leq\dbE[\Vert\xi\Vert^2_{H^k}]
   \period
   \label{EQ06}
  \end{equation}
We establish \eqref{EQ06} 
by induction on $k$, reasoning  by contradiction at each step. 
For $k=0$ (cf.~Step~1), we use \eqref{aldous-criterion} 
and the Aldous's criterion to obtain 
a compactness of measures induced by the process $\{u^n\}$. 
Then using \eqref{conserved-quantities} 
for a limiting measure we obtain a contradiction. 

{\it Step 1:} First we prove \eqref{EQ06} for $k=0$.
We assume that the convergence does not hold. This  means that, passing to a subsequence, 
there exists $\epsilon>0$ such that
  \begin{equation}\dbE[ \Vert u^n_{t_n}\Vert^2_{L^2}] \geq \dbE[ \Vert \xi\Vert^2_{L^2}]+\epsilon
     \comma n\in{\mathbb N}
   \label{EQ08}
   \period
  \end{equation}
We now pick $T>0$ such that 
$3 \gamma^{1/2} e^{-2\lambda T}\leq \epsilon$.
Note that, by \eqref{unif-bound-theo}, the sequence 
$\{u^n_{t_n-T}\}$ satisfies  
  \begin{equation} 
   \sup_{n} \dbE\left[\Vert u^n_{t_n-T}\Vert^4_{L^2}\right]\leq \gamma
   \period
  \end{equation}
Therefore, passing to a further subsequence, there exists 
an $L^2$-valued random variable $\xi_{-T}$ 
such that $u^n_{t_n-T}$ converges in distribution in 
$L^2_{\rm loc}(\dbR^d)$ to $\xi_{-T}$. 
Define a sequence of measures $\nu^n$ on $\cZ$ by
  \begin{equation}
      \nu^n(dz) 
        =\int_\Omega \delta_{\{\{u^n_{t_n-T+r}(\omega)\}_{r\in [0,T]}\}}(dz)\dbP(d\omega)
   \period
  \end{equation}
The assumption \eqref{aldous-criterion} and the Aldous criterion
\cite[Theorem~16.10]{B} imply that the sequence  
$\{\nu^n\}_{n=1}^{\infty}$ is tight in distribution in $\cZ$. 
Taking a further subsequence, 
we obtain the existence of $\nu$ such that 
  \begin{equation}
   \dbE^{\nu^n}\left[F(z)\right] =\dbE\left[F(\{u^n_{t_n-T+s}\}_{s\in[0,T]})\right] \to   \dbE^{\nu}\left[F(z)\right]\mbox{~~as~}n\to \infty
   \comma F\in C_{b}(\cZ)
   \period
  \end{equation}
Identifying the marginals, we easily see that the distribution of $z_T$ under $\nu$ is the same as the distribution of $\xi$. Similarly,  the distribution of $z_0$ under $\nu$ is the same as the distribution of $\xi_{-T}$.
We write the equation \eqref{conserved-quantities} at times $t_n$ and $t_n-T$ for the measure $\nu^n$
  \begin{equation}\label{norm-evolution-n}
  \dbE\left[\Vert u^n_{t_n}\Vert^2_{L^2}\right]-e^{-2\lambda T }\dbE\left[\Vert u^n_{t_n-T}\Vert_{L^2}^2\right]=\dbE\left[F_0(u^n_{t_n})-\tilde F_0(T,u^n_{t_n-T})\right]
+\int_0^T \dbE\left[G_0(T-r,u^n_{t_n-T+r})\right]dr\period
  \end{equation}
We claim that by the assumptions (i), (ii), and (a), 
we have sufficient integrability and continuity at the right hand side of the equation to use the convergence of $\nu^n$ to $\nu$ and pass to the limit to obtain 
  \begin{equation}\lim_n \Bigl(
                   \dbE\left[\Vert u^n_{t_n}\Vert^2_{L^2}\right]-e^{-2\lambda T }\dbE\left[\Vert u^n_{t_n-T}\Vert^2_{L^2}\right]
                         \Bigr)
      =\dbE^\nu\left[F_0(z_T)-\tilde F_0(T,z_0)\right]
+\int_0^T \dbE^\nu\left[G_0(T-r,z_r)\right]dr\period\end{equation}
Indeed, the convergence of $\nu^n$ to $\nu$ in $\cZ$ implies that for all $s\in [0,T]$ and every function $\xi\colon L^2_{\rm loc}(\dbR^d)\to \dbR$ continuous and bounded we have
  \begin{equation}\dbE\left[\xi(u^n_{t_n-T+s})\right]\to\dbE^\nu \left[\xi(z_s)\right]\period\end{equation} 
Note that, by the assumption (i), the mappings $F_0(\cdot)$, $\tilde F_0(T,\cdot)$, and $G_0(T-s,\cdot)$ are continuous in $L^2_{\rm loc}(\dbR^d)$ on bounded sets of $L^2(\dbR^d)$. Additionally the assumption (ii) and the uniform bound \eqref{unif-bound-theo} 
allows us to truncate $F_0(\cdot)$, $\tilde F_0(T,\cdot)$, and $G_0(T-s,\cdot)$ when they are large in order to obtain
  \begin{equation}\dbE\left[\psi(u^n_{t_n-T+s})\right]\to\dbE^\nu \left[\psi(z_s)\right]\end{equation} 
for $\psi=F_0(\cdot)$, $\psi=\tilde F(T,\cdot)$, and $\psi=G_0(T-s,\cdot)$. Thus
  \begin{equation}\lim_n \Bigl(\dbE\left[\Vert u^n_{t_n}\Vert^2_{L^2}\right]-e^{-2\lambda T }\dbE\left[\Vert u^n_{t_n-T}\Vert^2_{L^2}\right]\Bigr)=\dbE^\nu\left[F_0(z_T)-\tilde F_0(T,z_0)\right]
+\int_0^T \dbE^\nu\left[G_0(T-r,z_r)\right]dr\period\end{equation}
We shall show in  Step~3 that $\nu$ is a $L^2$-square integrable martingale solution of \eqref{EQ00}. Using this result and
by the assumption \eqref{conserved-quantities} one has
  \begin{equation}\dbE^\nu\left[F_0(z_T)-\tilde F_0(T,z_0)\right]
+\int_0^T \dbE^\nu\left[G_0(T-r,z_r)\right]dr=\dbE^\nu\left[\Vert z_T\Vert^2_{L^2}\right]-e^{-2\lambda T}\dbE^\nu\left[\Vert z_0\Vert^2_{L^2}\right]
   \period
\end{equation}
Noting the bound \eqref{unif-bound-theo}, we may pass to the limit and obtain
  \begin{align}
    &\lim_n \left(\dbE\left[\Vert u^n_{t_n}\Vert^2_{L^2}\right]-e^{-2\lambda T }\dbE\left[\Vert u^n_{t_n-T}\Vert^2_{L^2}\right]\right)
    \nonumber\\&\indeq
    =\dbE^\nu\left[\Vert z_T\Vert^2_{L^2}\right]-e^{-2\lambda T}\dbE^\nu\left[\Vert z_0\Vert^2_{L^2}\right]
    \nonumber\\&\indeq
   =\dbE\left[\Vert \xi\Vert^2_{L^2}\right]-e^{-2\lambda T}\dbE\left[\Vert \xi_{-T}\Vert^2_{L^2}\right]
   \period
   \label{EQ05}
  \end{align}
By convergence of $u_{t_n-T}^{n}$ to $\xi_{-T}$ in distribution in $L^2_{\rm loc}$ and the Fatou's lemma, we obtain
 \begin{equation}
      \dbE\left[\Vert \xi_{-T}\Vert^2_{L^2}\right]
   \leq \liminf_n \dbE\left[\Vert u^n_{t_n-T}\Vert^2_{L^2}\right] \leq \gamma^{1/2}
  \period\end{equation}
Using \eqref{EQ08} and \eqref{EQ05}, we obtain
  \begin{align}
    \epsilon &\leq 
    \liminf_n
      \dbE\left[\Vert u^n_{t_n}\Vert^2_{L^2}\right]-  \dbE\left[\Vert \xi\Vert^2_{L^2}\right]
   \nonumber\\&
    \leq \limsup_n e^{-2\lambda T }\left(\dbE
      \left[\Vert  u^n_{t_n-T}\Vert^2_{L^2}\right] 
      -\dbE\left[\Vert \xi_{-T}\Vert^2_{L^2}\right]\right)\leq \gamma^{1/2} e^{-2\lambda T } \leq \frac{2\epsilon}{3}
  \end{align}
which is a contradiction.

{\it Step 2:}  Now we prove that $\nu$ is a  $L^2$-square integrable martingale solution of \eqref{EQ00}.
Note that the uniform bound \eqref{unif-bound-theo}, the lower semicontinuity of the $L^2$ norm with respect to the $L^2_{\rm loc}$ topology, and the distributional convergence of $\nu^n$ to $\nu$ give that for all $t\in[0,T]$ 
  $$
   \dbE^\nu \left[\Vert z_t\Vert_{L^2}^2\right]\leq \gamma
   \period
  $$
Additionally the choice of the power for $b$ implies that the mapping $z\in\cZ \to M_t^{\phi}(z)$ is continuous. Thus for all $f$ bounded continuous 
  $$\lim_n \dbE^{\nu^n}[f(M_t^\phi)]= \dbE^{\nu}[f(M_t^\phi)]
   \period$$
For all $\phi$ smooth, there exists $K_{\phi,T}$ depending only on $\phi$ and $T$ such that 
\begin{align}\label{bound-mphit}
|M_t^\phi|^2 \leq  K_{\phi,T}\left(\Vert z_0\Vert^2_{L^2}+\Vert z_t\Vert^2_{L^2}+\int_0^t \Vert z_r\Vert^2_{L^2}+\Vert b(z_r)\Vert^2_{L^1} dr\right)
   \comma t\in[0,T]
   \period
\end{align}
We will use these points to prove that $M_t^\phi$ is a martingale under $\nu$.
We fix a family of smooth truncation functions $\Psi_R$ satisfying $|\Psi_R|\leq 2R$ and $\Psi_R(x)=x$ if $|x|\leq R$. For $0\leq s_1\leq s_2\leq\ldots\leq s_m\leq s\leq t\leq T$ smooth compactly supported functions functions $\phi_i$ and a random variable of the form $F=F(\int \phi_1 z_{s_1} dx,\ldots,\int \phi_m z_{s_m}dx)$ smooth and bounded by $1$, we have the equalities
\begin{align*}
\left | \dbE^\nu \left[ (M^{\phi}_t-M^\phi_s)F\right]\right|&= \left | \dbE^\nu \left[ (\Psi_R(M^{\phi}_t)-\Psi_R(M^\phi_s))F\right] \right| + \dbE^\nu \left[ |\Psi_R(M^{\phi}_t)-M^{\phi}_t|\right]+\dbE^\nu \left[ |\Psi_R(M^{\phi}_s)-M^{\phi}_s|\right]\\
&=\lim_n \left | \dbE^{\nu^n} \left[ (\Psi_R(M^{\phi}_t)-\Psi_R(M^\phi_s))F\right] \right| + \frac{1}{R}\left( \dbE^\nu \left[ |M^{\phi}_t|^2\right]+\dbE^\nu \left[ |M^{\phi}_s|^2\right]
\right) \\
&=\lim_n \left | \dbE^{\nu^n} \left[ (M^{\phi}_t-M^\phi_s)F\right] \right| + \frac{4T\gamma K_{\phi,T}}{R} 
   \period
\end{align*}
Note that by the martingale property of $M_t^\phi$ under $\nu^n$ we have $\dbE^{\nu^n} \left[ (M^{\phi}_t-M^\phi_s)F\right] =0$. Thus, taking $R$ to infinity we get
$ \dbE^\nu \left[ (M^{\phi}_t-M^\phi_s)F\right]=0$ which is sufficient to claim that $M_t^\phi$ is a $\nu$ martingale. 
Due to the smoothness of $\phi$, the continuity of $z$ in $L^2_{\rm
loc}$ and \eqref{bound-mphit}, $M_t^\phi$ is a continuous and square
integrable martingale under $\nu$. We now proceed to characterize its quadratic variation. 

By the definition of martingale solutions under $\nu^n$ the process $\fractext{M_t^\phi}{\sqrt{\sum_i (\Phi e_i,\phi)^2}}$ is a Brownian motion and thus has Gaussian independent increments. By the distributional convergence of $\nu^n$ and the continuity of $\fractext{M_t^\phi}{\sqrt{\sum_i (\Phi e_i,\phi)^2}}$ in the $L^2_{\rm loc}$ topology with respect to $z$ the distribution and the independence of the increments still hold under $\nu$. Thus the continuous process $\fractext{M_t^\phi}{\sqrt{\sum_i (\Phi e_i,\phi)^2}}$ is a Brownian motion under $\nu$ which implies that \eqref{EQ48} holds under $\nu$.

{\it Step 3:} For the induction step, assume that 
$\dbE[\Vert u^n_{t_n}\Vert^2_{H^r}]\to \dbE[\Vert \xi\Vert^2_{H^r}]$ 
as $n\to \infty$ for $r=0,\ldots,k-1$. 
We need to show that 
  \begin{equation}\dbE[\Vert u^n_{t_n}\Vert^2_{H^k}]\to \dbE[\Vert \xi\Vert^2_{H^k}]\period\end{equation}
Note that using Lemma~\ref{cv-loc-hi} at each step of 
the induction  one can also show that 
  \begin{equation}u^n_{t_n}\to \xi\end{equation} in distribution in $H^r(\dbR^d)$ for all $r\leq k-1$. 

In order to obtain a contradiction, 
assume that the convergence we are proving does not hold.
This means that, up to a subsequence,  there exists $\epsilon>0$ such that 
  \begin{equation}\dbE[ \Vert u^n_{t_n}\Vert^2_{H^k}] \geq \dbE[ \Vert \xi\Vert^2_{H^k}]+\epsilon
   \comma n\in{\mathbb N}\period\end{equation}
Similarly to the previous step, we introduce $T$ such that $3\gamma^{1/2} e^{-2\lambda T}\leq \epsilon$ and define the measures $\nu^n$ on $\cZ$. We also prove similarly that there exist an $H^k$-valued random variable $\xi_{-T}$, a distribution $\nu$ on $\cZ$ which is a $H^k$ square integrable solution of \eqref{EQ00} and a subsequence of $t_n$ (still denoted $t_n$) such that $\nu^n\to \nu$ on $\cZ$ as $n\to\infty$ and $u^n_{t_n-T}\to \xi_{-T}$ in distribution in $L^2_{\rm loc}(\dbR^d)$.  
Note that for all $s\in[0,T]$ the family $u^n_{t_n-T+s}$ converges in distribution in $L^2_{\rm loc}(\dbR^d)$ to the distribution of $z_s$ under $\nu$. Therefore, using the induction hypothesis on the family  $u^n_{t_n-T+s}$ and $z_s$ we obtain 
  \begin{equation}
    u^n_{t_n-T+s}\to z_s
  \end{equation}
in distribution in $H^r(\dbR^d)$
and
  \begin{equation}
    \dbE[\Vert u^n_{t_n-T+s}\Vert^2_{H^r}]\to \dbE^\nu [\Vert z_s\Vert^2_{H^r}]
  \end{equation}
for $r\leq k-1$ as $n\to \infty$.  

We first use \eqref{conserved-quantities} on $\nu^n$ for $k$ to obtain 
  \begin{equation}
   \dbE\left[\Vert u^n_{t_n}\Vert^2_{H^k}\right]-e^{-2\lambda T}\dbE\left[\Vert  u^n_{t_n-T}\Vert^2_{H^k}\right]=\dbE\left[F_k( u^n_{t_n})-\tilde F_k(T, u^n_{t_n})\right]
+\int_0^T \dbE\left[G_k(T-s, u^n_{t_n-T+s})\right]ds
   \period
\end{equation}
We have proven that the distribution of $u^n_{t_n-T+s}$ converges in distribution in $H^{k-1}(\dbR^d)$ to the distribution of $z_s$ under $\nu$.
Similarly to the previous step, we have enough integrability and continuity on the right hand side of the equation to use this convergence and pass to the limit to obtain
  \begin{equation}\lim_n \dbE\left[\Vert u^n_{t_n}\Vert^2_{H^k}\right]-e^{-2\lambda T}\dbE\left[\Vert  u^n_{t_n-T}\Vert^2_{H^k}\right]=\dbE^\nu\left[F_k(z_T)-\tilde F_k(T,z_0)\right]
+\int_0^T \dbE^\nu\left[G_k(T-s,z_s)\right]ds\period
  \end{equation}
Using Fatou's lemma and \eqref{unif-bound-theo} we have that $\nu$ is a $H^k(\dbR^d)$ square integrable solution of \eqref{EQ00}. Thus, by assumption, Definition~\ref{energy-evolution}~(iii) gives
  \begin{equation}\dbE^\nu\left[\Vert z_T\Vert^2_{H^k}\right]-e^{-2\lambda T}\dbE^\nu\left[\Vert  z_0\Vert^2_{H^k}\right]=\dbE^\nu\left[F_k(z_T)-\tilde F_k(T,z_0)\right]
+\int_0^T \dbE^\nu\left[G_k(T-s,z_s)\right]ds,
  \end{equation}
which implies 
  \begin{equation}\lim_n \Bigl(\dbE\left[\Vert u^n_{t_n}\Vert^2_{H^k}\right]-e^{-2\lambda T}\dbE\left[\Vert  u^n_{t_n-T}\Vert^2_{H^k}\right]\Bigr)=\dbE^\nu\left[\Vert z_T\Vert^2_{H^k}\right]-e^{-2\lambda T}\dbE^\nu\left[\Vert  z_0\Vert^2_{H^k}\right]\period
  \end{equation}
Using the same arguments as in the previous step, we obtain a contradiction. 
\end{proof}

\startnewsection{Proofs of tightness for the Schr\"odinger equation}{sec5}
We now return to the Schr\"odinger equation \eqref{equation-schro}.
We fix $\lambda>0$; thus all the constants are allowed to
depend on $\lambda$.
Also, recall that we impose 
Assumptions~\ref{assumption-existence-schro}
on $\sigma$ and $\Phi$.

\cole
\begin{Lemma}\label{bounds}
For every $k\in {\mathbb N}$
we have
%on $\lambda$ and $\Vert\Phi\Vert_{HS(L^2,H^1)}$ 
%such that 
  \begin{equation}\label{l2-bounds}
   \sup_{t\geq 0} \dbE[M(u(s))^{k}]
      \leq 
      C_k(\dbE[|M(u_0)|^k]+1)
  \end{equation}
and 
  \begin{equation}\label{h1-bounds}
   \sup_{t\geq 0} \dbE[H(u(s))^{k}]\leq C_k(\dbE[|H(u_0)|^k]+1)
  \end{equation}
where $C_k\ge0$ is a constant.
%for $k=1$ or $k=2$.
\end{Lemma}
\colb

\begin{proof}[Proof of Lemma~\ref{bounds}]
Using similar ideas as in \cite{EKZ}, one can show that the local martingale appearing in \eqref{evolution-M} is a martingale. Thus we have
  \begin{equation}\dbE[M(u(t))]+2\lambda \int_0^t \dbE[M(u(s))]ds =\dbE[M(u_0)]+ t \Vert\Phi\Vert^2_{{HS(L^2;L^)}}\period
  \end{equation}
Solving this ODE for $\dbE[M(u(t))]$, we get
  \begin{equation}\dbE[M(u(t))]=e^{-2\lambda t}\dbE[M(u_0)]+\Vert\Phi\Vert^2_{{HS(L^2; L^2)}}\int_0^t e^{-2\lambda (t-s)}ds\leq \dbE[M(u_0)]+\frac{1}{2\lambda} \Vert\Phi\Vert^2_{{HS(L^2; L^2)}}
  \end{equation}
which proves \eqref{l2-bounds} for $k=1$.
For general $k$ we proceed by induction. We assume the existence of $C_k$ for a given $k\geq 1$ and apply Ito's lemma to $M(u(t))^{k+1}$ to obtain 
\begin{align}
   &dM^{k+1}(u(t))+2(k+1)\lambda M^{k+1}(u(t))dt 
   \nonumber\\&\indeq
   = (k+1) M^{k}(u(t))\Vert\Phi \Vert_{{HS(L^2; L^2)}}^2 dt
   \nonumber\\&\indeq\indeq
     + \frac{k(k+1)}{2} M^{k-1}(u(t)) \sum_i Re(u(t),\Phi e_i)^2 dt+\tilde M
\end{align}
where similarly $\tilde M$ can be shown to be a martingale. 
Thus the function $\dbE\left[M^{k+1}(u(t))\right]$ satisfies the ODE
\begin{align}
    &\left(\dbE\left[M^{k+1}(u(t))\right]\right)'+2(k+1)\lambda\dbE\left[M^{k+1}(u(t))\right]
 \nonumber\\&\indeq
   = (k+1)\dbE\left[ M^{k}(u(t))\Vert\Phi \Vert_{{HS(L^2; L^2)}}^2\right]+ \frac{k(k+1)}{2} \dbE\left[M^{k-1}(u(t)) \sum_i Re(u(t),\Phi e_i)^2\right]=:g_k(t),
\end{align}
where $g_k$ is a bounded function of $t$ by the induction assumption. By solving this ODE, we see that  the function $\dbE\left[M^{k+1}(u(t))\right]$ is bounded.
Repeating the same argument for $H^k(u(t))$ we obtain \eqref{h1-bounds}.
\end{proof}

\nnewpage
In order to obtain the tightness of the averaged measures, we use 
Lemma~\ref{cv-loc-hi}
and
Theorem~\ref{thm-limit-conservation}.
The last ingredient we need is the following lemma. 

\cole
\begin{Lemma}\label{schr-aldous}
Under the assumptions of Lemma~\ref{lemma-tightness}, 
for all stopping times $T_n$ and  real numbers $\d_n$ 
such that $\d_n\to0$
as $n\to\infty$, we have
  \begin{equation}\label{lemma-aldous-criterion}
   \dbE\left[\Vert u^n_{T_n+\d_n}-u^n_{T_n}\Vert^2_{L^2}\right]\to 0\mbox{~~as~} n\to \infty\period
  \end{equation} 
\end{Lemma}
\colb

\begin{proof}[Proof of Lemma~\ref{schr-aldous}]
We denote by $S^\lambda$ the semigroup associated with the linear part of the equation. With this notation, we have
  \begin{align}
   u^n_{T_n+\d_n}-u^n_{T_n}
   &=
    S^\lambda(\d_n) u^n_{T_n}- u_{T_n}
     + i\int_{0}^{\d_n} S^\lambda(\d_n-s)(|u^n(T_n+s)|^{2\sigma}u^n(s)) ds 
       \nonumber\\&\indeq
   +\int_{0}^{\d_n} S^\lambda(\d_n-s)\Phi dW_{T_n+s}
   \period
  \end{align}
The lemma would follow from the following three convergence statements:
  \begin{align}
   &\dbE\left[\Vert S^\lambda(\d_n) u^n_{T_n}- u^n_{T_n}\Vert^2_{L^2}\right]\to 0,
   \label{EQ09}
  \\
   &\dbE\left[\Vert  \int_{0}^{\d_n} S^\lambda(\d_n-s)(|u^n(T_n+s)|^{2\sigma}u^n(s)) ds \Vert^2_{L^2}\right]\to 0,
   \label{EQ10} 
  \\
   &\dbE\left[\Vert \int_{0}^{\d_n} S^\lambda(\d_n-s)\Phi dW_{T_n+s}\Vert^2_{L^2}\right]\to 0
   \label{EQ11}
   \period
  \end{align}
By PDE arguments, the first convergence is obvious.
For the second convergence \eqref{EQ10}, we simply write 
  \begin{align}
    &\dbE\left[\left\Vert  \int_{0}^{\d_n} S^\lambda(\d_n-s)(|u(T_n+s)|^{2\sigma}u(s)) ds\right \Vert^2_{L^2}\right]
   \nonumber\\&\indeq
    \leq \d_n \int_{0}^{\d_n} \dbE\left[\Vert  S^\lambda(\d_n-s)(|u(T_n+s)|^{2\sigma}u(s))  \Vert^2_{L^2}\right]ds
   \period
  \end{align}
Given the uniform bounds \eqref{l2-bounds}, the integrand is uniformly bounded and
the convergence thus holds. 
For the third convergence \eqref{EQ11}, we use the the Burkholder-Davis-Gundy Inequality 
\cite[Lemma~5.24]{DZ} and obtain 
  \begin{equation}
   \dbE\left[\left\Vert \int_{0}^{\d_n} S^\lambda(\d_n-s)\Phi dW_{T_n+s}\right\Vert^2_{L^2}\right]\leq \int_0^{\d_n }\Vert S^\lambda(\d_n-s)\Phi \Vert_{HS(L^2,L^2)}^2 ds \to 0
  \end{equation}
as $n\to\infty$.
\end{proof}

\subsection{Proof of Lemma~\ref{lemma-tightness}}
{\it Proof of (i):}
We show that the assumptions of Theorem~\ref{thm-limit-conservation} with $k=1$ are satisfied for the equation~\eqref{equation-schro} and the set 
$    \{(\nu^nP_{t_n})(\cdot):n\in\dbN\}$ 
is relatively weakly compact over $H^1(\dbR^d)$.
We define 
\begin{align}
&F_0 =\tilde F_0=0
\nonumber\\&
G_0(t,r,v):=e^{-2\lambda (t-r)}\Vert \Phi\Vert_{HS(L^2,L^2)}^2,\\
&F_1(v)=\frac{1}{2\sigma +2} \int |v(x)|^{2\sigma+2} dx\\
&\tilde F_1(r,v):=e^{-2\lambda r}\frac{1}{2\sigma +2} \int |v(x)|^{2\sigma+2} dx,\\
&G_1(r,v):= e^{-2\lambda r}\biggl(\int_{\dbR^d} |v(x)|^{2\sigma+2} dx +\Vert \nabla \Phi\Vert^2_{HS(L^2,L^2)}
\nonumber\\&\indeq\indeq\indeq\indeq\indeq\indeq
- \Vert |v|^\sigma \Phi\Vert^2_{HS(L^2,L^2)}-\sigma \sum_i Re(|v|^{2\sigma -2 }v^2, (\Phi e_i)^2) \biggr)
   \period
\end{align}
We apply the Gagliardo-Nirenberg interpolation inequality to obtain 
\begin{equation}\Vert v\Vert_{L^{2\sigma+2}}^{2\sigma+2}\leq C\Vert v\Vert_{H^1}^{d\sigma}\Vert v\Vert^{\sigma(2-d)+2}_{L^2} \end{equation}
which shows that $F_1(\cdot)$, $\tilde F_1(r,\cdot)$, and $G_1(r,\cdot)$ are continuous in $L^2(\dbR^d)$ on bounded sets of $H^1(\dbR^d)$. They also have at most polynomial growth in $H^1(\dbR^d)$ and given the bounds on $u^n_0$ and Lemma~\eqref{bounds}, with $b(u)=|u|^{2\sigma}u$, we have the bound \eqref{unif-bound-theo}. Additionally, given the assumption \ref{assumption-existence-schro} on $\sigma$, we can easily verify that the degree of $b$ satisfy for $d\geq 2$, 
$$2\sigma +1 <\frac{d+2}{d-2}\leq \frac{2d}{d-2}.$$

Since $\nu$ is a $H^1$-square integrable martingale solution of \eqref{equation-schro}, by \cite[Theorem~2.4]{D}, we can extend the probability space $(\cZ,\cD,\nu)$ to obtain a family of Brownian motions $\hat B^i$ such that the $H^{-1}(\dbR^d)$-valued continuous martingale $ M_t= z(t)-z(0) +\int_0^t (\lambda z(s)-i\Delta z(s) -i|z(s)|^{2\sigma} z(s) )ds$ can be represented as 
\begin{equation}dM_t=\sum_i \Phi e_i d\hat B^i_t\period\end{equation}
Similarly to \cite[Propositions~3.2 and~3.3]{DD}, we apply Ito's lemma to $M(z(t))$ and $H(z(t))$ on this probability space to obtain that \eqref{conserved-quantities} holds for $i=0,1$ under $\nu$. This shows that the equation \eqref{equation-schro} has the $H^1$-norm evolution property.

Next, we consider the sequence of measures ${(\nu^nP_{t_n})(dv)}$ as measures on the space $L^2_{\rm loc}(\dbR^d)$. Denote by $B_k\subseteq \dbR^d$ the ball with radius $k$, centered at the origin. Given the uniform estimates \eqref{h1-bounds}, by the compact embedding of the space $H^1(B_k)$ in $L^2(B_k)$ and a successive application of Prokhorov's theorem, we obtain that there exists a subsequence of $\{t_n,u^n_0\}$, which we still denote $\{t_n,u^n_0\}$, and a distribution $\mu$ on $H^1(\dbR^d)$ such that 
\begin{equation}{(\nu^nP_{t_n})(dv)}\to\mu\mbox{~~in distribution in~}L^2_{\rm loc}(\dbR^d)\period\end{equation} 
We also note that the solutions $u^n$ satisfy the assumptions (iv) and (v) as consequences of \eqref{h1-bounds} and \eqref{lemma-aldous-criterion} respectively. 
Thus by 
Lemma~\ref{cv-loc-hi} 
and 
Theorem~\ref{thm-limit-conservation},
the convergence 
\begin{equation}{(\nu^nP_{t_n})(dv)}\to\mu\end{equation}
is in fact in distribution in $H^1(\dbR^d)$ which is what we claimed.

{\it Proof of (ii):}
We choose $(s_n,v_n)\in [0,T]\times K$. By the compactness of the two sets there exist a subsequence of $(s_n,v_n)$, still denoted $(s_n,v_n)$, and $(s,v)\in[0,1]\times K$ such that $(s_n,v_n)\to (s,v)$. We claim that $P_{s_n}(v_n,\cdot)$ converges in distribution in $H^1(\dbR^d)$ to $P_s(v,\cdot)$. 

We denote by $u^n$ and $u$ the solutions of \eqref{equation-schro} with initial data $v^n$ and $v$ respectively. In order to show this convergence we prove that we have
\begin{align}\label{conv-ii}\sup_{t\in[0,1]} \bigl(\Vert u^n_t -u_t\Vert_{H^1}+\Vert u_{s_n}-u_s\Vert_{H^1}\bigr)\to 0,\,\dbP\mbox{-a.s}.\end{align}
The convergence $\Vert u_{s_n}-u_s\Vert_{H^1}\to 0$ is a direct consequence of $u\in C([0,1],H^1),\,\dbP$-a.s. 
It is shown in \cite{DD} that 
  \begin{equation}
    \int_0^\cdot S^\lambda(\cdot-r)\Phi dW_r\in C([0,1];H^1(\dbR^d))\cap L^{\fractext{4(\sigma+1)}{\sigma d}}(0,1,W^{1,2\sigma +2}(\dbR^d))   
    \qquad
    \dbP\mbox{-a.s.}
   \label{EQ04}
  \end{equation}
Thus, applying  \cite[Proposition~3.5]{DD}, 
we also have $\dbP$-a.s. $|u^n-u|_{C([0,1];H^1(\dbR^d))}\to0$ as $n\to\infty$. 

We now show that \eqref{conv-ii} implies the convergence  $P_{s_n}(v^n,\cdot) \to P_s(v,\cdot)$. 
We pick $\xi\colon H^1(\dbR^d)\to\dbR$ uniformly continuous and bounded. Then 
\begin{align}
|P_s\xi (v)-P_{s_n}\xi(v^n)|&\leq \dbE[|\xi(u_s)-\xi(u^n_{s_n})|]
\nonumber
\\
&\leq  \dbE[|\xi(u_s)-\xi(u_{s_n})|]+ \dbE[|\xi(u_{s_n})-\xi(u^n_{s_n})|]
   \period
\end{align}
Note that \eqref{conv-ii} and the uniform continuity of $\xi$ imply that $\dbP$-a.s. $|\xi(u_s)-\xi(u_{s_n})|+|\xi(u_{s_n})-\xi(u^n_{s_n})|\to 0$ as $n\to\infty$. By the dominated convergence theorem, we obtain $|P_s\xi (v)-P_{s_n}\xi(v^n)|\to 0$ as $n\to \infty$.
\qed

\section{Compactness of the set of invariant measures}
In this section, we establish the existence
of an ergodic measure.

\cole
\begin{Theorem}
\label{T.compactness}
Under Assumptions~\ref{assumption-existence-schro}, 
the set of $H^1(\dbR^d)$-valued invariant measures is a convex and compact subset of 
the space of probability measures on $H^1(\dbR^d)$. 
\end{Theorem}
\colb

\begin{proof}
Note that the convexity is trivial,
so we only need to show compactness.
Let $\mu$ be such a measure and $(u(t))$ the solution of \eqref{equation-schro} having distribution $\mu$ at all time. For simplicity of notation, we denote $M_s=M(u(s))$ and $H_s=H(u(s))$. Our first objective is to prove the integrability of these semi-martingales.

We fix $R_0,R>0$ and define $\tau_{R}:=\inf\{s\geq 0: M_s\geq R\}$. We apply \eqref{evolution-M} on the event ${\{M_0\leq R_0\}}$ and obtain 
\begin{align}
M_{t\wedge \tau_R} =&M_0 e^{-2\lambda {t\wedge \tau_R}} +\Vert\Phi \Vert_{HS(L^2; L^2)}^2 \int_0^{t\wedge \tau_R} e^{-2\lambda({t\wedge \tau_R}-s)}  ds \\
&+2 \int_0^{t\wedge \tau_R} e^{-2\lambda({t\wedge \tau_R}-s)}Re(u(s),\Phi e_i) dB^i_s
   \period
\end{align}
Note that by the localization the expectation of the stochastic integral vanishes. Therefore,
\begin{equation}\dbE\left[M_{t\wedge \tau_R}\1_{\{M_0\leq R_0\}}\right] =\dbE\left[M_0\1_{\{M_0\leq R_0\}} e^{-2\lambda {t\wedge \tau_R}} \right]+\Vert\Phi \Vert_{HS(L^2; L^2)}^2\dbE\left[ \int_0^{t\wedge \tau_R} e^{-2\lambda({t\wedge \tau_R}-s)} \1_{\{M_0\leq R_0\}} ds\right]\period\end{equation}
For fixed $R_0$, the integrands on the right hand side are uniformly bounded and the integrand on the left hand side is non-negative. We apply the dominated convergence theorem for the right side and Fatou's lemma for the left to obtain that
\begin{equation}\dbE\left[M_t\1_{\{M_0\leq R_0\}} \right]\leq \dbE\left[M_0\1_{\{M_0\leq R_0\}}   \right]e^{-2\lambda t}+\frac{\Vert\Phi \Vert_{HS(L^2; L^2)}^2}{2\lambda}\period  \end{equation}
Therefore, we can choose $t_{R_0}>0$ such that for all ${R_0}>0$, we have
\begin{equation}\dbE\left[M_{t_{R_0}}\1_{\{M_0\leq {R_0}\}} \right]\leq \frac{\Vert\Phi \Vert_{HS(L^2; L^2)}^2}{\lambda}\period\end{equation}
Noting also that the distribution of $M_{t_{R_0}}$ is $\mu$ we obtain there exists $f_{R_0}(v) \to 1$ $\mu$-a.s. as ${R_0}\to\infty$ and 
\begin{equation}\dbE\left[M_{t_{R_0}}\1_{\{M_0\leq {R_0}\}} \right]=\int  \Vert v\Vert_{L^2}^2f_{R_0}(v)\mu(dv)\leq \frac{\Vert\Phi \Vert_{HS(L^2; L^2)}^2}{\lambda}\period\end{equation}
Taking the limit ${R_0}\to \infty$, we obtain 
  \begin{equation}
   \int  \Vert v\Vert_{L^2}^2\mu(dv)\leq \frac{\Vert\Phi \Vert_{HS(L^2; L^2)}^2}{\lambda}
   \period
   \label{EQ07}
  \end{equation}
Similarly to the proof of Lemma~\ref{bounds}, we apply  Ito's lemma to $M^{k+1}(u(t))$, localize with stopping times and prove that there exists $C_k(\Phi,\lambda)$ which may a~priori depend on $\mu$ such that
\begin{equation}\int \Vert v\Vert_{L^2}^{2k}\mu (dv)\leq C_k(\Phi,\lambda)<\infty,
   \comma  k=1\ldots\period\end{equation}

We also apply the same procedure to $H_t$ to obtain that there exists $\tilde C_k (\Phi,\lambda)$ that may again depend on $\mu$ such that 
\begin{equation}\int \Vert v\Vert_{H^1}^{2k}\mu (dv)\leq \tilde C_k(\Phi,\lambda)<\infty\period\end{equation}
Given this integrability, we return to \eqref{evolution-M} and\eqref{evolution-H} to prove that $C_k(\Phi,\lambda)$ and $\tilde C_k (\Phi,\lambda)$ can be taken independent of $\mu$. Since $\mu$ is an invariant measure, we get $d \dbE[M_t]=d\dbE[H_t]=0$ and 
\begin{equation}\dbE[M_t]=  \frac{\Vert\Phi \Vert_{HS(L^2; L^2)}^2}{2\lambda}\period\end{equation}
Using the same invariance we obtain 
\begin{align}
2\lambda \dbE[M^{k+1}_t]&= \Vert \Phi\Vert_{HS(L^2,L^2)}^2  \dbE[M^{k}_t] +\frac{k}{2} \dbE[M^{k-1}_t\sum_i Re (u(t),\Phi e_i)^2]\\
&\leq \left(\Vert \Phi\Vert_{HS(L^2,L^2)}^2 + \frac{k}{2}\right) \dbE[M^{k}_t],
\end{align}
which shows by induction that $C_k(\Phi,\lambda)$ may be taken independent of $\mu$. 
Applying the same procedure to the equation \eqref{evolution-H}, we obtain that $\tilde C_k(\Phi,\lambda)$ can be taken independent of $\mu$. 

We now prove the sequential compactness of the set of $H^1(\dbR^d)$-valued invariant measures. Let $\mu^n$ be a sequence of such invariant measures of the equation \eqref{equation-schro}. 
Without loss of generality, we assume that the $\sigma$-algebra $\dbF_0$ is rich enough so that there exists a family of $\dbF_0$-measurable random variables $u^n_0$ with distribution $\mu^n$. 
The uniform bounds we have proven gives us
\begin{equation}\sup_n \int \Vert v\Vert_{L^2}^{2k}\mu^n (dv)\leq C_k(\Phi,\lambda)
\end{equation}
and
\begin{equation}
\sup_n \int  \Vert v\Vert_{H^1}^{2}\mu^n (dv)\leq \tilde C_k (\Phi,\lambda),\end{equation}
which a~fortiori imply
\begin{equation}\dbE\left[\Vert u^n_0\Vert _{H^1}^{4\vee\lceil{4d\sigma}\rceil}+\Vert u^n_0\Vert _{L^2}^{\lceil{4\sigma(2-d)+8}\rceil}\right]\leq R\period\end{equation}
Therefore, Lemma~\ref{lemma-tightness} and the fact that $\mu^n$ is an invariant measure show that the family 
\begin{equation}\bigl\{(\mu^n P_{t_n})(\cdot):n\in\dbN\bigr\}=\{\mu^n:n\in\dbN\}\end{equation} is tight. Noting that the set of invariant measures is closed, we obtain the required compactness. 
\end{proof}

\cole
\begin{Corollary}
\label{c.ergodic}
Under Assumptions~\ref{assumption-existence-schro}, there exists an ergodic invariant measure.
\end{Corollary}
\colb

\begin{proof}
By the Krein-Milman theorem, the compactness of the set of invariant measures implies that there exists at least one invariant measure that is an extremal point of this set. Proposition 3.2.7 of \cite{DZ} then implies that such a measure is ergodic. 
\end{proof}

\section*{Acknowledgments} 
I.K.~was supported in part by the NSF grant DMS-1311943,
%I.K.~was supported in part by the NSF grant DMS-1311943 and DMS-1615239,
while M.Z.~was supported in part by the NSF grant DMS-1109562.


\begin{thebibliography}{[GKVZ]}
\bibitem[B]{B} %MR1766371
J.M. Ball, \emph{Continuity properties and global attractors of generalized
  semiflows and the {N}avier-{S}tokes equations%, [ {MR}1462276 (98j:58071a),b]
  },
  Mechanics: from theory to computation, Springer, New York, 2000,
  pp.~447--474. 
  %\MR{1766371}

\bibitem[Bi]{Bi}%{MR1700749}
P.~Billingsley, \emph{Convergence of probability measures}, second ed.,
  Wiley Series in Probability and Statistics: Probability and Statistics, John
  Wiley \& Sons, Inc., New York, 1999, A Wiley-Interscience Publication.
  %\MR{1700749 (2000e:60008)}

%\bibitem[BS1]{BS1} %MR0385355
%J.~L. Bona and R.~Smith, \emph{The initial-value problem for the {K}orteweg-de
%  {V}ries equation}, Philos. Trans. Roy. Soc. London Ser. A \textbf{278}
%  (1975), no.~1287, 555--601. 
%  %\MR{0385355 (52 \#6219)}

%\bibitem[BS2]{BS2} %MR0352736
%J. Bona and R. Smith, \emph{Existence of solutions to the {K}orteweg-de
%  {V}ries initial value problem}, Nonlinear wave motion ({P}roc. {AMS}-{SIAM}
%  {S}ummer {S}em., {C}larkson {C}oll. {T}ech., {P}otsdam, {N}.{Y}., 1972),
%  Amer. Math. Soc., Providence, R.I., 1974, pp.~179--180. Lectures in Appl.
%  Math., Vol. 15. 
%  %\MR{0352736 (50 \#5223)}

%\bibitem[Bo]{Bo} %{MR0276436}
%N.~Bourbaki, \emph{\'{E}l\'ements de math\'ematique. {F}asc. {XXXV}. {L}ivre
%  {VI}: {I}nt\'egration. {C}hapitre {IX}: {I}nt\'egration sur les espaces
%  topologiques s\'epar\'es}, Actualit\'es Scientifiques et Industrielles, 
%  No.~1343, Hermann, Paris, 1969. %\MR{0276436 (43 \#2183)}

%\bibitem[Bou]{Bou} %MR1466164
%J.~Bourgain, \emph{Periodic {K}orteweg de {V}ries equation with measures as
%  initial data}, Selecta Math. (N.S.) \textbf{3} (1997), no.~2, 115--159.
%  %\MR{1466164 (2000i:35173)}

%\bibitem[CKSTT]{CKSTT} %MR1969209
%J.~Colliander, M.~Keel, G.~Staffilani, H.~Takaoka, and T.~Tao, \emph{Sharp
%  global well-posedness for {K}d{V} and modified {K}d{V} on {$\Bbb R$} and
%  {$\Bbb T$}}, J. Amer. Math. Soc. \textbf{16} (2003), no.~3, 705--749
%  (electronic). 
%  %\MR{1969209 (2004c:35352)}

\bibitem[CGV]{CGV} %MR3223489
P. Constantin, N. Glatt-Holtz, and V. Vicol, \emph{Unique ergodicity
  for fractionally dissipated, stochastically forced 2{D} {E}uler equations},
  Comm. Math. Phys. \textbf{330} (2014), no.~2, 819--857. 
  %\MR{3223489}

%\bibitem[Da]{Da} %MR1775236
%T.~Dankel, Jr., \emph{On the stochastic {K}orteweg-de {V}ries equation driven
%  by white noise}, Differential Integral Equations \textbf{13} (2000), no.~7-9,
%  827--836. 
%  %\MR{1775236 (2001g:35279)}

\bibitem[DZ]{DZ}%{MR1417491}
G.~Da~Prato and J.~Zabczyk, \emph{Ergodicity for infinite-dimensional systems},
  London Mathematical Society Lecture Note Series, vol. 229, Cambridge
  University Press, Cambridge, 1996. 
  %\MR{1417491 (97k:60165)}

%\bibitem[DD1]{DD1}%{MR1616536}
%A.~de~Bouard and A.~Debussche, \emph{On the stochastic {K}orteweg-de {V}ries
%  equation}, J. Funct. Anal. \textbf{154} (1998), no.~1, 215--251. 
%  %\MR{1616536 (99c:35209)}

\bibitem[DD]{DD} %MR2310695
A.~de~Bouard and A.~Debussche. \emph{The stochastic nonlinear Schr\"odinger equation in $H^1$}, Stochastic Analysis and Applications,Vol. 21, Iss. 1, (2003).
%\MR{2310695 (2008i:60103)}

%\bibitem[DD3]{DD3} %MR2532941
%A. de~Bouard and A. Debussche, \emph{On a stochastic {K}orteweg-de
%  {V}ries equation with homogeneous noise}, S\'eminaire: \'{E}quations aux
%  {D}\'eriv\'ees {P}artielles. 2007--2008, S\'emin. \'Equ. D\'eriv. Partielles,
%  \'Ecole Polytech., Palaiseau, 2009, pp.~Exp. No. V, 15. 
%  %\MR{2532941 (2011b:60259)}

%\bibitem[DD4]{DD4} %MR2223907
%A. Debussche and J. Printems, \emph{Convergence of a semi-discrete
%  scheme for the stochastic {K}orteweg-de {V}ries equation}, Discrete Contin.
%  Dyn. Syst. Ser. B \textbf{6} (2006), no.~4, 761--781 (electronic).
%  %\MR{2223907 (2007j:60103)}

\bibitem[DO]{DO} %MR2174876
A. Debussche and C. Odasso, \emph{Ergodicity for a weakly damped
stochastic non-linear {S}chr\"odinger equation}, J. Evol. Equ. \textbf{5}
(2005), no.~3, 317--356. 
%\MR{2174876 (2006g:37076)}

%\bibitem[DP]{DP} %MR1711301
%A. Debussche and J. Printems, \emph{Numerical simulation of the
%  stochastic {K}orteweg-de {V}ries equation}, Phys. D \textbf{134} (1999),
%  no.~2, 200--226. 
%  %\MR{1711301 (2000g:76015)}

\bibitem[DV]{DV} %MR2652180
A.~Debussche and J.~Vovelle, \emph{Scalar conservation laws with stochastic
forcing}, J. Funct. Anal. \textbf{259} (2010), no.~4, 1014--1042. 
%\MR{2652180   (2012h:60201)}

\bibitem[D]{D} %MR1105550
E. Dettweiler, \emph{Representation of {B}anach space valued martingales as
stochastic integrals}, Probability in {B}anach spaces, 7 ({O}berwolfach,
1988), Progr. Probab., vol.~21, Birkh\"auser Boston, Boston, MA, 1990,
pp.~43--62. 
%\MR{1105550}

\bibitem[EKZ]{EKZ}
I. Ekren, I. Kukavica, and M. Ziane, \emph{Existence of
invariant measures for the stochastic damped KdV equation} (2015)
arXiv:1512.02686.


\bibitem[F1]{F1} %MR1300150
F. Flandoli, \emph{Dissipativity and invariant measures for stochastic
  {N}avier-{S}tokes equations}, NoDEA Nonlinear Differential Equations Appl.
  \textbf{1} (1994), no.~4, 403--423. 
  %\MR{1300150 (95h:35254)}

\bibitem[F2]{F2} %MR2459085
F. Flandoli, \emph{An introduction to 3{D} stochastic fluid dynamics},
  S{PDE} in hydrodynamic: recent progress and prospects, Lecture Notes in
  Math., vol. 1942, Springer, Berlin, 2008, pp.~51--150. 
  %\MR{2459085 (2009j:76191)}

%\bibitem[G]{G} %MR952903
%J.-M. Ghidaglia, \emph{Weakly damped forced {K}orteweg-de {V}ries
%  equations behave as a finite-dimensional dynamical system in the long time},
%  J. Differential Equations \textbf{74} (1988), no.~2, 369--390. 
%  %\MR{952903 (90b:35205)}

\bibitem[GKVZ]{GKVZ}
  N.~Glatt-Holtz, I.~Kukavica, V.~Vicol, and M.~Ziane, Existence
  and regularity of invariant measures for the
  three dimensional stochastic primitive equations
  {J.~Math.\ Phys.}~\textbf{55} (2014), 34~pp.

\bibitem[GMR]{GMR}
  N.~Glatt-Holtz, J.~Mattingly, and G.~Richards,
  On unique ergodicity in nonlinear stochastic partial differential
  equations,
  arXiv:~1512.04126v1.

\bibitem[G1]{G1} %MR1623435
O. Goubet, \emph{Regularity of the attractor for a weakly damped nonlinear
  {S}chr\"odinger equation}, Appl. Anal. \textbf{60} (1996), no.~1-2, 99--119.
  %\MR{1623435}

\bibitem[G2]{G2} %MR1751948
O. Goubet, \emph{Regularity of the attractor for a weakly damped nonlinear
  {S}chr\"odinger equation in {${\bf R}^2$}}, Adv. Differential Equations
  \textbf{3} (1998), no.~3, 337--360. 
  %\MR{1751948}

\bibitem[G3]{G3} %Go} %MR175739
O. Goubet, \emph{Asymptotic smoothing effect for weakly damped forced
  {K}orteweg-de {V}ries equations}, Discrete Contin. Dynam. Systems \textbf{6}
  (2000), no.~3, 625--644. 
  %\MR{1757391 (2001e:35019)}

\bibitem[GM]{GM} %MR1488375
O.~Goubet and I.~Moise, \emph{Attractor for dissipative {Z}akharov system},
  Nonlinear Anal. \textbf{31} (1998), no.~7, 823--847. 
  %\MR{1488375 (99a:35259)}

\bibitem[GK]{GK} %MR2754293
O. Goubet and W. Kechiche, \emph{Uniform attractor for non-autonomous
  nonlinear {S}chr\"odinger equation}, Commun. Pure Appl. Anal. \textbf{10}
  (2011), no.~2, 639--651. 
  %\MR{2754293}

\bibitem[GL]{GL} %MR2985677
O. Goubet and L. Legry, \emph{Existence of a finite-dimensional
  global attractor for a damped parametric nonlinear {S}chr\"odinger equation},
  Adv. Differential Equations \textbf{17} (2012), no.~9-10, 859--877.
  %\MR{2985677}

\bibitem[GR]{GR} %MR1935630
O. Goubet and R.M.S. Rosa, \emph{Asymptotic smoothing and the
  global attractor of a weakly damped {K}d{V} equation on the real line}, J.
  Differential Equations \textbf{185} (2002), no.~1, 25--53. 
  %\MR{1935630 (2003j:35027)}

\bibitem[HM]{HM} %MR2259251
M. Hairer and J.C. Mattingly, \emph{Ergodicity of the 2{D}
  {N}avier-{S}tokes equations with degenerate stochastic forcing}, Ann. of
  Math. (2) \textbf{164} (2006), no.~3, 993--1032. 
  %\MR{2259251}

\bibitem[KS]{KS} %MR1860883
S. Kuksin and A. Shirikyan, \emph{Ergodicity for the randomly forced
  2{D} {N}avier-{S}tokes equations}, Math. Phys. Anal. Geom. \textbf{4} (2001),
  no.~2, 147--195. 
  %\MR{1860883}

%\bibitem[K]{K} %MR759907
%T. Kato, \emph{On the {C}auchy problem for the (generalized) {K}orteweg-de
%  {V}ries equation}, Studies in applied mathematics, Adv. Math. Suppl. Stud.,
%  vol.~8, Academic Press, New York, 1983, pp.~93--128. 
%  %\MR{759907 (86f:35160)}

%\bibitem[KPV1]{KPV1} %MR1086966
%C.E. Kenig, G. Ponce, and L. Vega, \emph{Well-posedness of the
%  initial value problem for the {K}orteweg-de {V}ries equation}, J. Amer. Math.
%  Soc. \textbf{4} (1991), no.~2, 323--347. 
%  %\MR{1086966 (92c:35106)}

%\bibitem[KPV2]{KPV2} %MR1211741
%C.E. Kenig, G. Ponce, and L. Vega, \emph{Well-posedness and
%  scattering results for the generalized {K}orteweg-de {V}ries equation via the
%  contraction principle}, Comm. Pure Appl. Math. \textbf{46} (1993), no.~4,
%  527--620. 
%  %\MR{1211741 (94h:35229)}

%\bibitem[Ku1]{Ku1} %MR2433681
%S.B. Kuksin, \emph{Eulerian limit for 2{D} {N}avier-{S}tokes equation and
%  damped/driven {K}d{V} equation as its model}, Tr. Mat. Inst. Steklova
%  \textbf{259} (2007), no.~Anal. i Osob. Ch. 2, 134--142. 
%  %\MR{2433681 (2009f:76045)}

%\bibitem[Ku]{Ku} %MR2738999
%S.B. Kuksin, \emph{Damped-driven {K}d{V} and effective equations for
%  long-time behaviour of its solutions}, Geom. Funct. Anal. \textbf{20} (2010),
%  no.~6, 1431--1463. 
%  %\MR{2738999 (2012a:37160)}

%\bibitem[L]{L} %MR1614084
%P. Lauren{\c{c}}ot, \emph{Compact attractor for weakly damped driven
%  {K}orteweg-de {V}ries equations on the real line}, Czechoslovak Math. J.
%  \textbf{48(123)} (1998), no.~1, 85--94. 
%  %\MR{1614084 (99a:35224)}

%\bibitem[MR]{MR} %MR1424770
%I.~Moise and R.~Rosa, \emph{On the regularity of the global attractor of a
%  weakly damped, forced {K}orteweg-de {V}ries equation}, Adv. Differential
%  Equations \textbf{2} (1997), no.~2, 257--296. 
%  %\MR{1424770 (97i:35155)}
  
\bibitem[MR]{MR} %MR2399706
C.M. Mora and R.~Rebolledo, \emph{Basic properties of nonlinear
  stochastic {S}chr\"odinger equations driven by {B}rownian motions}, Ann.
  Appl. Probab. \textbf{18} (2008), no.~2, 591--619. 
  %\MR{2399706}

\bibitem[MiR1]{MiR1} %MR1661767
R.~Mikulevicius and B.L. Rozovskii, \emph{Martingale problems for stochastic
  {PDE}'s}, Stochastic partial differential equations: six perspectives, Math.
  Surveys Monogr., vol.~64, Amer. Math. Soc., Providence, RI, 1999,
  pp.~243--325. 
  %\MR{1661767}

\bibitem[MiR2]{MiR2} %MR2118862
R.~Mikulevicius and B.L. Rozovskii, \emph{Global {$L_2$}-solutions of
  stochastic {N}avier-{S}tokes equations}, Ann. Probab. \textbf{33} (2005),
  no.~1, 137--176. 
  %\MR{2118862}


%\bibitem[O]{O} %MR2603800
%T. Oh, \emph{Periodic stochastic {K}orteweg-de {V}ries equation with
%  additive space-time white noise}, Anal. PDE \textbf{2} (2009), no.~3,
%  281--304. 
%  %\MR{2603800 (2011c:60208)}


\bibitem[P]{P}%MR1683626
J. Printems, \emph{The stochastic {K}orteweg-de {V}ries equation in
  {$L^2(\bold R)$}}, J. Differential Equations \textbf{153} (1999), no.~2,
  338--373. 
  %\MR{1683626 (2001g:35280)}

\bibitem[Pr]{Pr}
Y.~Prokhorov, \emph{Convergence of random processes and 
limit theorems in probability theory}, Theory of Prob. and Appl.~\textbf{2} (1956), no.~2,157-214.

\bibitem[R]{R} %{MR1812879}
R. Rosa, \emph{The global attractor of a weakly damped, forced
  {K}orteweg-de {V}ries equation in {$H^1(\Bbb R)$}}, Mat. Contemp. \textbf{19}
  (2000), 129--152, VI Workshop on Partial Differential Equations, Part II (Rio
  de Janeiro, 1999). 
  %\MR{1812879 (2001m:35274)}


%\bibitem[ST]{ST} %MR0454425
%J.~C. Saut and R.~Temam, \emph{Remarks on the {K}orteweg-de {V}ries equation},
%  Israel J. Math. \textbf{24} (1976), no.~1, 78--87. 
%  %\MR{0454425 (56 \#12676)}

\bibitem[T]{T} %MR0261183
R.~Temam, \emph{Sur un probl\`eme non lin\'eaire}, J. Math. Pures Appl. (9)
  \textbf{48} (1969), 159--172. 
  %\MR{0261183 (41 \#5799)}

\end{thebibliography}
\end{document}